\documentclass{amsart}

\usepackage{amsfonts,amssymb,amscd,amsmath,enumerate,verbatim,calc}

\newcommand{\tra}{\operatorname{Tr}}

\newcommand{\BZ}{{\boldsymbol Z}}

\def\Magma{{\sc Magma }}

\newcommand{\tr}{\operatorname{Tr}}
\theoremstyle{plain}
\newtheorem{theorem}{Theorem}

\newtheorem{lemma}[theorem]{Lemma}

\newtheorem{proposition}[theorem]{Proposition}

\theoremstyle{definition}

\theoremstyle{remark}
\newtheorem{remark}[theorem]{Remark}

\newcounter{hours}\newcounter{minutes}
\newcommand{\printtime}{%
        \setcounter{hours}{\time/60}%
        \setcounter{minutes}{\time-\value{hours}*60}%
        \thehours\,h\ \theminutes\,min}

\begin{document}

\title[Klein four group and cyclic groups]
{Separating invariants for the klein four group and cyclic groups}
\date{\today,\ \printtime}
\author{ Martin Kohls}
\address{Technische Universit\"at M\"unchen \\
 Zentrum Mathematik-M11\\
Boltzmannstrasse 3\\
 85748 Garching, Germany}
\email{kohls@ma.tum.de}

\author{M\"uf\.it Sezer}
\address { Department of Mathematics, Bilkent University,
 Ankara 06800 Turkey}
\email{sezer@fen.bilkent.edu.tr}
\thanks{We thank  Gregor Kemper for funding a visit of the second author to TU M\"{u}nchen and T\" {u}bitak for
funding a visit of the first author to Bilkent University. Second
author is also partially supported by T\"{u}bitak-Tbag/109T384 and
T\"{u}ba-Gebip/2010. }

\begin{abstract}
We consider indecomposable representations of the Klein four group
over a field of characteristic $2$ and of a cyclic group of order
$pm$ with $p,m$ coprime over a field of characteristic $p$. For
each representation we explicitly describe a separating set in the
corresponding ring of invariants. Our construction is recursive
and the separating sets we obtain consist of almost entirely orbit
sums and products.
\end{abstract}

 \maketitle

\section{introduction}
Let $V$ be a finite dimensional representation of a group $G$ over
an algebraically closed field $F$. In the sequel we
 will also call $V$ a $G$-module. There is an induced action on the symmetric
algebra $F[V]:=S(V^{*})$ given by $\sigma (f)=f\circ \sigma^{-1}$ for $\sigma
\in G$ and $f\in F[V]$. We let $F[V]^G$ denote the
 subalgebra of invariant polynomials in $F[V]$. A subset
$A\subseteq F[V]^G$ is said to be separating for $V$ if for any
pair of vectors $u,w\in V$, we have: If $f(u)=f(w)$ for all $f\in
A$, then $f(u)=f(w)$ for all $f\in F[V]^G$. Goals in
invariant theory include finding generators and studying properties of
invariant rings. In the
study of separating invariants the goal is rather to find and
describe a subalgebra of the ring of invariants which separates
the group orbits. 
Although separating invariants have been object of study since the
early times of invariant theory, they have regained particular
attention following the influential textbook of  Derksen and
Kemper \cite{{DerksenKemper}}. The invariant ring is often too
complicated and it is difficult to describe explicit generators
and relations. Meanwhile, there have been several papers within
the last decade that demonstrate that one can construct separating
subalgebras with nice  properties that make them more
accessible. For instance 
Noether's (relative) bound holds for separating invariants
independently of the characteristic of the field \cite[Corollary
3.9.14]{DerksenKemper}. For more results on separating algebras we
direct the reader to
\cite{MR2308028,MR2414957,DufresneSeparating,MRdek,MRdk,MRkadish,KohlsKraft,Mr10}.

If the order of the group is divisible by the characteristic of
the field, then the degrees of the generators increase unboundedly
as the dimension of the representation increases. Therefore
computing the invariant ring in this case is particularly
difficult. Even in the simplest situation of a cyclic group of
prime order acting through Jordan blocks, explicit generating sets
are known only for a handful of cases. This rather short list of
cases consists of  indecomposable representations  up to dimension
nine and decomposable ones whose indecomposable summands have
dimension at most four. See \cite{MR1639884} for a classical work
and \cite{weh} for the most recent advances in this matter which
also gives a good taste of the difficulty of the problem. On the
other hand separating invariants for these representations have a
surprisingly simple theory. In \cite{MR2556140,Comb} it is observed
that a separating set for a indecomposable representation of a cyclic
$p$-group over a field of characteristic $p$ can be obtained by
adding some explicitly defined invariant polynomials to a separating set for a
certain quotient representation.  The
main ingredient of the proofs of these results is the efficient
use of the surjection of a representation to a quotient representation
to establish a link between the respective separating sets that
generating sets do not have.
  In this paper we build on this technique to construct separating invariants for
  the indecomposable representations of the Klein four group over a field of
characteristic $2$ and  of a cyclic group of order $pm$ with $p,m$
coprime over a field of characteristic $p$. Despite being the
immediate follow ups of the cyclic $p$-groups, their invariant
rings have not been computed yet. Therefore these groups (and
representations) appear to be the natural cases to consider. As in
the case for cyclic $p$-groups, we describe a finite separating
set recursively. We remark that in \cite[Theorem
3.9.13]{DerksenKemper}, see also \cite[Corollary 19]{MR0000001}, a
way is given for calculating separating invariants explicitly for
any finite group. This is done by presenting a large polynomial
whose coefficients form a separating set. On the other hand, the
separating sets we compute consist of invariant polynomials that
are almost exclusively orbit sums and products. These are
``basic'' invariants  which are easier to obtain. Additionally,
our approach respects the inductive
 structure of the considered modules. Also, the
size of the set we give for the cyclic group of order $pm$ depends
only on the dimension of the representation while the size in
\cite[Theorem 3.9.13]{DerksenKemper} depends on the group order as well.
Hence, for large $p$ and $m$ our separating set is much smaller
for this group.

The strategy of our construction is based on the following
theorem.

 \begin{theorem}
  \label{ilk}
 Let $V$ and $W$ be  $G$-modules,  $\phi:V\rightarrow W$ a $G$-equivariant surjection, and $\phi^{*}:
  F[W]\hookrightarrow F[V]$ the corresponding inclusion.
  Let $S\subseteq F[W]^G$ be a separating set for $W$ and let $T\subseteq F[V]^G$ be a set of invariant polynomials such
  that if $v_1, v_2\in V$ are in different $G$-orbits and if
  $\phi(v_1)=\phi(v_2)$, then there is a polynomial $f\in T$ such
  that $f(v_1)\neq f(v_2)$. Then $\phi^{*}(S)\cup T$ is a separating set for
  $V$.
 \end{theorem}
 \begin{proof}
 Pick two vectors $v_1, v_2\in V$ in different $G$-orbits.
 If $\phi (v_1)$ and $\phi (v_2)$ are in different $G$-orbits, then there exits a
 polynomial $f\in S$ that separates these vectors, so $\phi^{*}(f)$
 separates $v_{1},v_{2}$. So we may assume that
 $\phi (v_1)$
 and $\phi (v_2)$ are in the same $G$-orbit. Furthermore, by replacing
 $v_2$ with a suitable vector in its orbit we may take $\phi (v_1)=\phi
 (v_2)$. Hence, by construction, $T$ contains an invariant that separates $v_1$
 and $v_2$ as desired.
 \end{proof}
Before we finish this section we recall the definitions of a
transfer and a norm. For a subgroup $H\subseteq G$ and $f\in
F[V]^H$, the relative transfer $\tra^G_H(f)$ is defined to be
$\sum_{\sigma \in G/H}\sigma (f)$. We also denote
$\tra^G_{\{\iota\}}(f)=\tra^G (f)$, where $\iota$ is the identity
element of $G$. Also for $f\in F[V]$, the norm  $N_H(f)$ is
defined to be  the product $\prod_{\sigma \in H}\sigma (f)$.

\section{ The Klein four group }

For the Klein four group
$G=\{\iota,\sigma_{1},\sigma_{2},\sigma_{3}\}$ over an
algebraically closed field $F$ of characteristic $2$, the complete
list of indecomposable $G$-modules  is given in Benson
\cite[Theorem 4.3.3]{BensonCohomology1}. For each module in the
list, we will explicitly construct  a finite separating set. The
modules in this list come  in five ``types''. We use the same
enumeration as in \cite{BensonCohomology1}. The first type (i) is
just the regular representation $FG$ of $G$, and a separating set
or even the invariant ring can be computed with \Magma
\cite{Magma}. In the following, we will thus concentrate on the
remaining four types, where each type consists of an infinite
series of indecomposable representations. Let $I_{n}$ denote the
identity matrix of $F^{n\times n}$, and $J_{\lambda}$ denote an
upper triangular Jordan block of size $n$ with eigenvalue
$\lambda\in F$. Let $H_{i}=\{\iota,\sigma_{i}\}$ for $i=1,2,3$ be
the three subgroups of order $2$.

\subsection{Type (ii)}
Let $G$ act on $V_{2n}=F^{2n}$ by the representation
$\sigma_{1}\mapsto \left(\begin{array}{cc} I_{n}&I_{n}\\ 0& I_{n}
  \end{array}\right)$ and $\sigma_{3}\mapsto\left(\begin{array}{cc} I_{n}&J_{\lambda}\\ 0& I_{n}
  \end{array}\right)$. We write $F[V_{2n}]=F[x_{1},\ldots,x_{2n}]$. We then
have
\[
\begin{array}{rclcl}
\sigma_{1}x_{i}&=&x_{i}+x_{n+i}  &\text{ for }&1\le i\le n,\\
\sigma_{3}x_{i}&=&x_{i}+\lambda x_{n+i}+x_{n+i+1} &\text{ for
}&1\le i\le
n-1,\\
\sigma_{3} x_{n}&=&x_{n}+\lambda x_{2n},\\
x_{n+i}&\in&F[V_{2n}]^{G}& \text{ for }&1\le i\le n.
\end{array}
\]
We  start by computing several transfers and norms modulo some
subspaces of $F[V_{2n}]$. Define $R:=F[x_{2},\ldots,x_{n}]$. Note
that $S:=F[x_{1},\ldots,x_{n-1},x_{n+1},\ldots,x_{2n}]$ is a
$G$-subalgebra of $F[V_{2n}]$, and the congruence in Lemma
\ref{trxixj}(a) also holds modulo $S\cap
R=F[x_{2},\ldots,x_{n-1},x_{n+1},\ldots,x_{2n}]$.  This will be
needed for type (v), so we mark this result with a star.

\begin{lemma}\label{trxixj} We have
\begin{enumerate}
\item[(a*)] $\tra^G(x_1x_ix_j)\equiv
x_1(x_{n+i}x_{n+j+1}+x_{n+i+1}x_{n+j})
  \mod R$ for $2\le i,j\le n-1$.
\item[(b)] $\tra^{G}(x_{1}x_{n-1}x_{n})\equiv x_{1}x_{2n}^{2}\mod
R$.
\end{enumerate}
\end{lemma}

\begin{proof}
(a*) We only have to care for terms containing $x_{1}$, so
\begin{eqnarray*}
\tra^G(x_1x_ix_j)&\equiv& x_1x_ix_j+x_1(x_i+x_{n+i})(x_j+x_{n+j})\\
&&+
x_{1}(x_{i}+\lambda x_{n+i}+  x_{n+i+1})(x_{j}+\lambda x_{n+j} +x_{n+j+1})\\&&+x_{1}(x_{i}+(\lambda+1)x_{n+i}+x_{n+i+1})(x_{j}+(\lambda+1)x_{n+j}+x_{n+j+1})\\
&\equiv& x_1x_{n+i}x_{n+j+1}+x_1x_{n+i+1}x_{n+j} \mod R.
\end{eqnarray*}
(b) Follows from above with $i=n-1,\,\,j=n$ and setting
$x_{2n+1}:=0$.
\end{proof}

\begin{lemma}\label{trx1x23lambda}
For $n\ge 3$ we have
\begin{enumerate}
\item[(a)] $\tra^{G}(x_{1}x_{2}^{3})\equiv
\lambda(\lambda+1)x_{1}x_{n+2}^{3} \mod (R+x_{n+3}F[V_{2n}]).$
\item[(b)] For $\lambda\in\{0,1\}$, we have the invariant
\[N_{H_{2}}(x_1x_{n+2}+x_{2}x_{n+1})\equiv x_1^2x_{n+2}^2+
x_{1}x_{n+2}(x_{n+2}^{2}+x_{n+1}x_{n+3}) \mod R.\]
\end{enumerate}
\end{lemma}

\begin{proof}
(a) We only care for terms containing $x_{1}$ and not $x_{n+3}$,
so
\begin{eqnarray*}
\tra^{G}(x_{1}x_{2}^{3})&\equiv&x_{1}x_{2}^{3}+x_{1}(x_{2}+x_{n+2})^{3}+x_{1}(x_{2}+\lambda x_{n+2})^{3}\\
&&+x_{1}(x_{2}+(\lambda+1) x_{n+2})^{3}\\
&\equiv&\lambda(\lambda+1)x_{1}x_{n+2}^{3} \mod
(R+x_{n+3}F[V_{2n}]).
\end{eqnarray*}
(b) Just note that $x_1x_{n+2}+x_{2}x_{n+1}$ is $H_{1}$-invariant,
so the norm is $G$-invariant.
\end{proof}

Let $(a_1, \dots , a_{n}, a_{n+1}, \dots , a_{2n})\in F^{2n}$. We
have a $G$-equivariant surjection $V_{2n}\rightarrow V_{2n-2}$
given by $$\phi: (a_1, \dots , a_{n}, a_{n+1}, \dots , a_{2n})
\rightarrow (a_2, \dots , a_{n}, a_{n+2}, \dots , a_{2n})\in
F^{2n-2}.$$ Therefore $F[V_{2n-2}]=F[x_2, \cdots , x_{n}, x_{n+2},
\cdots x_{2n}]$ is a $G$-subalgebra of $F[V_{2n}]=F[x_1, \cdots ,
x_{n}, x_{n+1}, \cdots ,x_{2n}].$

\begin{proposition}
Let $n\ge 3$ and $S\subseteq F[V_{2n-2}]^G$ be a separating set
for $V_{2n-2}$. Then $\phi^{*}(S)$ together with the set $T$
consisting of \[x_{n+1}, \quad N_G(x_1), \quad
f_{\lambda}:=\left\{\begin{array}{ll} \tra^{G}(x_{1}x_{2}^{3})
    &\text{ for }\lambda\ne 0,1\\N_{H_{2}}(x_1x_{n+2}+x_{2}x_{n+1})&\text{ for
    }\lambda\in \{0,1\}\end{array}\right.
\]
  \[\tra^G(x_1x_{i}x_{i+1})\,\, \text{ for }
2\le i\le n-1,\]
 is a separating set for $V_{2n}$.
\end{proposition}

\begin{proof}
Let $v_1=(a_1, \dots , a_{n}, a_{n+1}, \dots , a_{2n} )$ and
$v_2=(b_1, \dots , b_{n}, b_{n+1}, \dots , b_{2n} )$ be two
vectors in $V_{2n}$ in different $G$-orbits with
$\phi(v_{1})=\phi(v_{2})$, so  $a_i=b_i$ except for  $i=1, n+1$.
To  apply Theorem \ref{ilk}, we assume for a contradiction that
all elements of $T$ take the same values on $v_{1}$ and $v_{2}$.
Since $x_{n+1}\in T$, we have $a_{n+1}=b_{n+1}$, hence we have
$v_2=(b_1, a_2, \dots , a_{n}, a_{n+1}, \dots , a_{2n} )$. Because
of Lemma \ref{trxixj} (b) we can assume $a_{2n}=0$. Since
$\tra^{G}(x_1x_{i}x_{i+1})\equiv
x_{1}(x_{n+i}x_{n+i+2}+x_{n+i+1}^{2}) \mod R$ for $2\le i\le n-2$,
we succesively get $a_{2n-1}=a_{2n-2}=\ldots=a_{n+3}=0$. In case
$\lambda\ne 0,1$ we can also assume $a_{n+2}=0$ by Lemma
\ref{trx1x23lambda}(a). In case $\lambda\in\{0,1\}$ and
$a_{n+2}\ne 0$, $N_{H_{2}}(x_1x_{n+2}+x_{2}x_{n+1})$ taking the
same value on $v_{1},v_{2}$ implies $a_{1}=b_{1}+a_{n+2}$, hence
$v_{1}=\sigma_{3}v_{2}$ for $\lambda=0$ and
$v_{1}=\sigma_{2}v_{2}$ for $\lambda=1$ respectively. So now
assume $a_{n+2}=0$. Then $N_{G}(x_{1})(v_{1})=N_{G}(x_{1})(v_{2})$
implies $a_{1}+b_{1}\in\{a_{n+1},\lambda
a_{n+1},(\lambda+1)a_{n+1}\}$, hence $v_{1}=\sigma_{i}v_{2}$ for
some $i\in\{1,2,3\}$.
\end{proof}

We give the induction start for $n=2$ and $\lambda\ne 0,1$ - the
case $\lambda\in\{0,1\}$ is left to the reader (or to \Magma).

\begin{lemma}
A separating set for $\lambda\ne 0,1$ and $n=2$ is given by the
invariants
\[
f_{1}:=x_{1}x_{4}+\frac{1}{\lambda(\lambda+1)}x_{2}^{2}+x_{2}(x_{3}+\frac{1}{\lambda(\lambda+1)}x_{4}),
\]
\[
N_{G}(x_{1}), \quad N_{G}(x_{2}), \quad  x_{3}, \quad x_{4}
\]
\end{lemma}

Note that since $G$ is not a reflection group, we need at least
$5$ separating invariants by \cite[Theorem
1.1]{DufresneSeparating}.

\begin{proof}
The invariants $x_{3},x_{4}$ allow us to consider two points
$v_{1}=(a _{1},a_{2},a_{3},a_{4})$ and
$v_{2}=(b_{1},b_{2},a_{3},a_{4})$ in different orbits. If
$N_{G}(x_{2})(v_{1})=N_{G}(x_{2})(v_{2})$, then
$a_{2}+b_{2}\in\{0,a_{4},\lambda a_{4},(\lambda+1)a_{4}\}$, so
after replacing $v_{2}$ by an element in its orbit we can assume
$a_{2}=b_{2}$. If $a_{4}\ne 0$, then $f_{1}$ separates
$v_{1},v_{2}$, so assume $a_{4}=0$. Then
$N_{G}(x_{1})(v_{1})=N_{G}(x_{1})(v_{2})$ implies
$a_{1}+b_{1}\in\{0,a_{3},\lambda a_{3},(\lambda+1)a_{3}\}$, so
$v_{1},v_{2}$ are in the same orbit.
\end{proof}

\subsection{ Type (iii) } Let $G$ act on $V_{2n}=F^{2n}$ by the representation
$\sigma_{1}\mapsto\left(\begin{array}{cc} I_{n}&J_{0}\\ 0& I_{n}
  \end{array}\right)$ and $\sigma_{3}\mapsto \left(\begin{array}{cc} I_{n}&I_{n}\\ 0& I_{n}
  \end{array}\right)$. This leads to the same invariants as in type (ii) with $\lambda=0$, just
$\sigma_{1}$ and $\sigma_{3}$ are interchanged.

\subsection{ Type (iv) }Take $\lambda=1$ in type (ii). If
$e_{1},\ldots,e_{2n}\in F^{2n}$ denotes the standard basis
vectors, we can consider the submodule $V_{2n-1}:=\langle
e_{1},\ldots,e_{n},e_{n+2},\ldots, e_{2n}\rangle$. Its
representation is given by \[\sigma_{1}\mapsto
\left(\begin{array}{c|c} I_{n}&  \begin{array}{c} 0_{n-1} \\\hline
I_{n-1}\end{array}  \\ \hline 0& I_{n-1}
  \end{array}\right)\quad\text{and}\quad \sigma_{2}\mapsto
\left(\begin{array}{c|c} I_{n}&  \begin{array}{c} I_{n-1} \\\hline
0_{n-1}\end{array}  \\ \hline 0& I_{n-1}
  \end{array}\right),
\]
which correpsonds to type (iv). Since the restriction map
$F[V_{2n}]^{G}\rightarrow F[V_{2n-1}]^{G}$, $f\mapsto
f|_{V_{2n-1}}$ maps separating sets to separating sets by
\cite[Theorem 2.3.16]{DerksenKemper}, we are done by our treatment
of type (ii).

 \subsection{ Type (v)}
Again we look at the case $\lambda=1$ of type (ii). Then $\langle
 e_{n}\rangle$ is a $G$-submodule, and we look at the factor module
 $V_{2n-1}:=V_{2n}/\langle e_{n}\rangle$ with basis
 $\tilde{e_{i}}:=e_{i}+\langle e_{n}\rangle$,
 $i\in\{1,\ldots,2n\}\setminus\{n\}$. Its representation is given by
\[\sigma_{1}\mapsto
\left(\begin{array}{c|c} I_{n-1}&  \begin{array}{c|c}
I_{n-1}&0_{n-1} \end{array}  \\ \hline 0& I_{n}
  \end{array}\right)\quad\text{and}\quad \sigma_{2}\mapsto
\left(\begin{array}{c|c} I_{n-1}&  \begin{array}{c|c}
0_{n-1}&I_{n-1}\end{array}  \\ \hline 0& I_{n}
  \end{array}\right).
\]
We have a $G$-algebra inclusion
$F[V_{2n-1}]=F[x_{1},\ldots,x_{n-1},x_{n+1},\ldots,x_{2n}]\subset
F[V_{2n}]$.

The action on the variables is given by
\[
 \sigma_1(x_i)=\left\{\begin{array}{ll} x_i+x_{n+i} &\text{ for } 1\le i\le n-1 \\
       x_i &\text { for } n+1\le i\le 2n, \end{array}\right.
\]

and
\[
 \sigma_2(x_i)=\left\{\begin{array}{ll} x_i+x_{n+i+1} &\text{ for } 1\le i\le n-1 \\
       x_i &\text { for } n+1\le i\le 2n. \end{array}\right.
\]
Let $(a_1, \dots , a_{n-1}, a_{n+1}, \dots , a_{2n})\in
F^{2n-1}\cong V_{2n-1}$. We have a $G$-equivariant surjection
$V_{2n-1}\rightarrow V_{2n-3}$ given by $$\phi: (a_1, \dots ,
a_{n-1}, a_{n+1}, \dots , a_{2n}) \rightarrow (a_2, \dots ,
a_{n-1}, a_{n+2}, \dots , a_{2n})\in F^{2n-3}.$$ Therefore
$F[V_{2n-3}]=F[x_2, \cdots , x_{n-1}, x_{n+2}, \cdots x_{2n}]$ is
a $G$-subalgebra of $F[V_{2n-1}]=F[x_1, \cdots , x_{n-1}, x_{n+1},
\cdots ,x_{2n}].$
 Also, let $R:=F[x_2, \cdots , x_{n-1}, x_{n+1}, \cdots
x_{2n}]$. We will make computations modulo $R$, considered as a
subvectorspace of $F[V_{2n-1}]$, and we can re-use the equation of
Lemma \ref{trxixj}(a*).

\begin{lemma}
\label{ikili} Let $v_1, v_2\in V_{2n-1}$ be two vectors  in
different orbits  that agree everywhere except the first
coordinate.
 Say,  $v_1=(a_1,
\dots , a_{n-1}, a_{n+1}, \dots , a_{2n} )$, $v_2=(b_1, a_2 \dots
, a_{n-1}, a_{n+1}, \dots , a_{2n} )$. Assume further that one of
the following holds.
\begin{enumerate}
\item[(a)] $a_{n+2}\neq 0$ and $a_i=0$ for $n+3\le i\le 2n$
\item[(b)] $a_i=a_{2n}\neq 0$ for $n+2\le i\le 2n-1$.
\end{enumerate}
Then the invariant \[ f:=N_{H_2}(x_1x_{n+2}+x_{2}x_{n+1})\equiv
x_1^2x_{n+2}^2+ x_{1}x_{n+2}(x_{n+2}^{2}+x_{n+1}x_{n+3}) \mod R
\]
 separates
$v_1$ and $v_2$.
\end{lemma}

\begin{proof}
Assume the first case. Then $f(v_{1})=f(v_{2})$ implies
 $(a_{1}+b_{1})^{2}a_{n+2}^{2}=(a_{1}+b_{1})a_{n+2}^{3}$, hence $a_{1}=b_{1}+a_{n+2}$.
 Since $a_i=0$ for $i\ge n+3$ this implies that $v_{1}=\sigma_2v_2$
 which is a contradiction because $v_1$ and $v_2$ are in different
 orbits.

Next assume the second case. Then $f(v_1)=f(v_2)$ implies
$(a_{1}+b_{1})^{2}a_{n+2}^{2}=(a_{1}+b_{1})a_{n+2}^{2}(a_{n+1}+a_{n+2})$,
hence $a_{1}=b_{1}+a_{n+1}+a_{n+2}$. Since $a_i=a_{2n}$ for
$n+2\le i\le
 2n-1$, this implies that $v_1=\sigma_3v_2$ yielding a
 contradiction.
\end{proof}

\begin{lemma}\label{Trx1xi3}
For $2\le i\le n-1$, we have the following elements in
$F[V_{2n-1}]^{G}$:
\[\tra^{G}(x_{1}x_{i}^{3})\equiv x_{1}x_{n+i}x_{n+i+1}(x_{n+i}+x_{n+i+1})\mod
R.\]
\end{lemma}

\begin{proof}
\begin{eqnarray*}
\tra^{G}(x_{1}x_{i}^{3})&\equiv&x_{1}x_{i}^{3}+x_{1}(x_{i}+x_{n+i})^{3}+x_{1}(x_{i}+x_{n+i+1})^{3}\\
&&+x_{1}(x_{i}+x_{n+i}+x_{n+i+1})^{3}\\
&\equiv&x_{1}x_{n+i}x_{n+i+1}(x_{n+i}+x_{n+i+1}) \mod R.
\end{eqnarray*}
\end{proof}

\begin{proposition}
Let $n\ge 3$ and $S\subseteq F[V_{2n-3}]^G$ be a separating set
for $V_{2n-3}$. Then $\phi^{*}(S)$ together with the set $T$
consisting of \[x_{n+1}, \quad N_G(x_1),\quad
 N_{H_2}(x_1x_{n+2}+x_{2}x_{n+1}),\quad \tra^{G}(x_{1}x_{2}x_{n-1}),\]  \[\tra^G(x_1x_{i}x_{i+1})\,\, \text{ for }
2\le i\le n-2,\quad\quad \tra^G(x_{1}x_{i}^{3})  \text{ for } 2\le
i\le n-1,\]
 is a separating set for $V_{2n-1}$.
\end{proposition}

\begin{proof}
Let $v_1=(a_1, \dots , a_{n-1}, a_{n+1}, \dots , a_{2n} )$ and
$v_2=(b_1, \dots , b_{n-1}, b_{n+1}, \dots , b_{2n} )$ be two
vectors in $V_{2n-1}$ in different $G$-orbits with
$\phi(v_{1})=\phi(v_{2})$, so  $a_i=b_i$ except for  $i=1, n+1$.
To  apply Theorem \ref{ilk}, we assume for a contradiction that
all elements of $T$ take the same values on $v_{1}$ and $v_{2}$.
Since $x_{n+1}\in T$, we have $a_{n+1}=b_{n+1}$, hence we have
$v_2=(b_1, a_2, \dots , a_{n-1}, a_{n+1}, \dots , a_{2n} )$.

We first assume $a_{n+i}\ne 0$ for $2\le i\le n$. Lemma
\ref{Trx1xi3} implies $a_{n+2}=a_{n+3}=\ldots=a_{2n}\ne 0$, a
contradiction to Lemma \ref{ikili} (b). Thus there must be a $2\le
i\le n$ with $a_{n+i}=0$, so let $i$ be maximal with this
property. Consider the invariants
$f_{j}:=\tra^{G}(x_{1}x_{j}x_{j+1})\equiv
x_{1}(x_{n+j}x_{n+j+2}+x_{n+j+1}^{2}) \mod R$ of $T$ for $2\le
j\le n-2$ (see Lemma \ref{trxixj}(a*)).

If $i\le n-2$, then $a_{n+i+1}\ne 0$, and $f_{i}$ separates
$v_{1},v_{2}$.

If $i=n-1$, then $a_{2n}\ne 0$, and $f_{j}(v_{1})=f_{j}(v_{2})$
for $j=n-3,n-4,\ldots,2$ implies $a_{n+j}=0$ for $3\le j\le n-1$.
As $\tra(x_{1}x_{2}x_{n-1})\equiv
x_{1}(x_{n+2}x_{2n}+x_{n+3}x_{2n-1})\mod R$ takes the same value
on $v_{1},v_{2}$, we also have $a_{n+2}=0$. Now
$N_{G}(x_{1})(v_{1})=N_{G}(x_{1})(v_{2})$ implies
$a_{1}=b_{1}+a_{n+1}$, thus $v_{1}=\sigma_{1}v_{2}$.

If $i=n$, i.e. $a_{2n}=0$, then since  $f_{j}(v_{1})=f_{j}(v_{2})$
for $j=n-2,n-3,\ldots,2$, we get $a_{n+j}=0$ for $3\le j\le 2n$.
In case $a_{n+2}\ne 0$, we are done by Lemma \ref{ikili} (a). If
$a_{n+2}=0$, then $N_{G}(x_{1})(v_{1})=N_{G}(x_{1})(v_{2})$
implies as before $a_{1}=b_{1}+a_{n+1}$ and
$v_{1}=\sigma_{1}v_{2}$.
\end{proof}

\begin{remark}
A separating set for $V_{3}$ is formed by
$N_{G}(x_{1}),\,\,\,x_{3},\,\,\,\,x_{4}$.
\end{remark}

 \section{ Cyclic groups}
Let $G=\BZ_{p^{r}m}$ be the cyclic group of order $p^{r}m$, where
$p$ is a prime number and $r,m$ are non-negative integers with
$(m,p)=1$. Let $H$ and $M$ be the subgroups of $G$ of order
$p^{r}$ and $m$, respectively. Let $V_{n}$ be an indecomposable
$G$-module of dimension $n$.

\begin{lemma}
There exists a basis $e_1, e_2, \dots , e_n$ of $V_n$ such that
$\sigma^{-1} (e_i)=e_i+e_{i+1}$ for $1\le i\le n-1$ and
$\sigma^{-1} (e_n)=e_n$ for a generator $\sigma$ of $H$, and
$\alpha (e_i)=\lambda e_i$ for $1\le i\le n$ for a $m$-th root of
unity $\lambda\in F$ and $\alpha$ a generator of $M$.
\end{lemma}

\begin{proof}
It is well known that $n\le p^{r}$ and there is basis such that a
generator $\rho$ of $G$ acts by a Jordan matrix $J_{\mu}=\mu
I_{n}+N$ with $\mu$ a $m$th root of unity \cite[p. 24]{Alperin}.
Then $\rho^{p^{r}}$ is a generator of $M$ acting by $(\mu
I_{n}+N)^{p^{r}}=\mu^{p^{r}}I_{n}$, and $\rho^{m}$ is a generator
of $H$ acting by $(\mu I_{n}+N)^{m}=I_{n}+m\mu^{m-1}N+{m\choose
2}\mu^{m-2}N^{2}+\ldots$. This matrix has Jordan normal form $J_{1}$,
and the representation matrix of $\rho^{p^{r}}$ is invariant under
base change, which proves the lemma.
\end{proof}

Since we want our representation to be faithful, we will assume
that $\lambda$ is a primitive $m$th root of unity from now. We
also restrict to the case $r=1$. Let $x_1, x_2, \dots , x_n$ be
the corresponding basis elements in $V_n^*$. We have $\sigma
(x_i)=x_i+x_{i-1}$ for $2\le i\le n$, $\sigma (x_1)=x_1$ and
$\alpha (x_i)=\lambda^{-1} x_i$ for $1\le i\le n$.     Since
$\alpha$ acts by multiplication by a primitive $m$th root of
unity, there exists a non-negative integer $k$  such that
$x_nx_{i+1}^{p-1}x_i^{k}\in F[V_n]^{M}$ for $1\le i \le n-2$. We
may assume that $k$ is the smallest such integer.
 Let $I_i$ denote the ideal in $F[V_n]$ generated by $x_{1},x_{2},\ldots,x_{i}$.
 Set $f_i=x_nx_{i+1}^{p-1}x_i^{k}$ for $1\le i\le n-2$.

\begin{lemma} Let $a$ be a positive integer. Then
$\sum_{0\le l\le p-1}l^{a}\equiv -1 \, \mod p$ if $p-1$ divides
$a$ and $\sum_{0\le l\le p-1}l^{a}\equiv 0 \, \mod p$, otherwise.
\end{lemma}
\begin{proof}
See \cite[9.4]{MR1424447} for a proof for this statement.
\end{proof}
Now set $R:=F[x_1, x_2, \dots ,x_{n-1} ]$.

\begin{lemma} Let  $1\le i\le n-2$.
We have $$\tr^G_M(f_i)\equiv -x_nx_i^{p+k-1}  \, \mod \big
(I_{i-1}+R \big ).$$
\end{lemma}

\begin{proof}
We only care for terms containing $x_{n}$ but not
$x_{1},\ldots,x_{i-1}$, thus
we have \begin{alignat*}{1}\sigma^l (f_i)&=(x_n+lx_{n-1}+{l \choose 2}x_{n-2}+ \cdots )(x_{i+1}+lx_{i}+ \cdots )^{p-1}(x_{i}+lx_{i-1}+  \cdots )^{k}\\
&\equiv x_n(x_{i+1}+lx_{i})^{p-1}x_{i}^{k} \mod \big (I_{i-1}+R
\big ).
\end{alignat*}
Thus it suffices to show that 
$\sum_{0\le l\le p-1}(x_{i+1}+lx_{i})^{p-1}=-x_i^{p-1}$. Let $a$
and $b$ be non-negative integers such that $a+b=p-1$. Then the
coefficient of $x_{i+1}^ax_{i}^b$ in $ (x_{i+1}+lx_{i})^{p-1}$ is
${p-1 \choose b}l^b$ and so the coefficient of $x_{i+1}^ax_{i}^b$
in $\sum_{0\le l\le p-1}(x_{i+1}+lx_{i})^{p-1}$ is $\sum_{0\le
l\le p-1}{p-1 \choose b}l^b$. Hence the result follows from the
previous lemma.
\end{proof}

Let $(c_1, c_2, \dots c_n)$ be a  vector in $V_n$. There is a
$G$-equivariant surjection $\phi: V_n\rightarrow V_{n-1}$ given by
$(c_1, c_2, \dots c_n)\rightarrow (c_1, c_2, \dots c_{n-1})$.
Hence $F[V_{n-1}]=F[x_1, \cdots, x_{n-1}]$ is a $G$-subalgebra of
$F[V_n]$. Let $l$ be the smallest non-negative integer such that
$N_H(x_n)(N_H(x_{n-1}))^l\in F[V_n]^G$. Note that since $(p,m)=1$
such an integer exists.

\begin{proposition}
Let $S\subseteq F[V_{n-1}]^G$ be a separating set for $V_{n-1}$.
Then $\phi^{*}(S)$ together with the set $T$ consisting of
\[N_H(x_n)(N_H(x_{n-1}))^l,\quad N_G(x_n),\quad \tr^G_M(f_i) \quad
\text{  for  } 1\le i\le n-2,\] is a separating set for $V_n$.
\end{proposition}

\begin{proof}
Let $v_1=(c_1, c_2, \dots, c_n)$ and  $v_2=(d_1, d_2, \dots, d_n)$
be two vectors in $V_n$ in different $G$-orbits with
$\phi(v_{1})=\phi(v_{2})$, so  $c_i=d_i$ for  $1\le i\le n-1$. To
apply Theorem \ref{ilk}, we assume for a contradiction that all
elements of $T$ take the same values on $v_{1}$ and $v_{2}$.
 Assume that there exists an integer $i\le
n-2$ such that $c_i\neq 0$. Assume further that $i$ is the
smallest such integer. Then $\tr^G_M(f_i)$ separates $v_1$ and
$v_2$ by the previous lemma. Therefore we have $c_1=c_2=\cdots
=c_{n-2}=0$. We consider two cases. First assume that $c_{n-1}=0$.
Then $N_G(x_n)(v_1)=N_G(x_n)(v_2)$, i.e. $c_{n}^{pm}=d_{n}^{pm}$,
implies that $c_n=\lambda^a d_n$ for some integer $a$ and hence
$v_1$ and $v_2$ are in the same orbit.
 If
 $c_{n-1}\neq 0$, then we see that $N_H(x_n)(N_H(x_{n-1}))^l$
separates $v_1$ and $v_2$ as follows.
 We have
 $(N_H(x_{n-1}))^l(v_1)=(N_H(x_{n-1}))^l(v_2)\neq 0$. Therefore it
 suffices to show
 $N_H(x_n)(v_1)\neq N_H(x_n)(v_2)$. But otherwise
 $c_{n}^{p}-c_{n}c_{n-1}^{p-1}=d_{n}^{p}-d_{n}c_{n-1}^{p-1}$, which implies
 $c_n=d_n+lc_{n-1}$ for some $0\le l\le p-1$, so $v_1$ and $v_2$
 are in the same orbit.
\end{proof}

\bibliographystyle{plain}
\bibliography{Bibliography_Version_10}

\begin{thebibliography}{10}

\bibitem{Alperin}
J.~L. Alperin.
\newblock {\em Local representation theory}, volume~11 of {\em Cambridge
  Studies in Advanced Mathematics}.
\newblock Cambridge University Press, Cambridge, 1986.
\newblock Modular representations as an introduction to the local
  representation theory of finite groups.

\bibitem{BensonCohomology1}
D.~J. Benson.
\newblock {\em Representations and cohomology. {I}}, volume~30 of {\em
  Cambridge Studies in Advanced Mathematics}.
\newblock Cambridge University Press, Cambridge, second edition, 1998.
\newblock Basic representation theory of finite groups and associative
  algebras.

\bibitem{Magma}
Wieb Bosma, John Cannon, and Catherine Playoust.
\newblock The {M}agma algebra system. {I}. {T}he user language.
\newblock {\em J. Symbolic Comput.}, 24(3-4):235--265, 1997.
\newblock Computational algebra and number theory (London, 1993).

\bibitem{MR1424447}
H.~E.~A. Campbell, I.~P. Hughes, R.~J. Shank, and D.~L. Wehlau.
\newblock Bases for rings of coinvariants.
\newblock {\em Transform. Groups}, 1(4):307--336, 1996.

\bibitem{DerksenKemper}
Harm Derksen and Gregor Kemper.
\newblock {\em Computational invariant theory}.
\newblock Invariant Theory and Algebraic Transformation Groups, I.
  Springer-Verlag, Berlin, 2002.
\newblock Encyclopaedia of Mathematical Sciences, 130.

\bibitem{MR2308028}
M.~Domokos.
\newblock Typical separating invariants.
\newblock {\em Transform. Groups}, 12(1):49--63, 2007.

\bibitem{MR2414957}
Jan Draisma, Gregor Kemper, and David Wehlau.
\newblock Polarization of separating invariants.
\newblock {\em Canad. J. Math.}, 60(3):556--571, 2008.

\bibitem{DufresneSeparating}
Emilie Dufresne.
\newblock Separating invariants and finite reflection groups.
\newblock {\em Adv. Math.}, 221(6):1979--1989, 2009.

\bibitem{MRdek}
Emilie Dufresne, Jonathan Elmer, and Martin Kohls.
\newblock The {C}ohen-{M}acaulay property of separating invariants of finite
  groups.
\newblock {\em Transform. Groups}, 14(4):771--785, 2009.

\bibitem{MRdk}
Emilie Dufresne and Martin Kohls.
\newblock A finite separating set for {D}aigle and {F}reudenburg's
  counterexample to {H}ilbert's {F}ourteenth {P}roblem.
\newblock {\em To appear in Comm. Algebra, arXiv:0912.0638v1}, 2009.

\bibitem{MRkadish}
Harlan Kadish.
\newblock Polynomial bounds for invariant functions separating orbits.
\newblock {\em Preprint, arXiv:1005.3082v1}, 2010.

\bibitem{MR0000001}
G.~Kemper.
\newblock Separating invariants.
\newblock {\em J. Symbolic Comput.}, 44:1212--1222, 2009.

\bibitem{KohlsKraft}
Martin Kohls and Hanspeter Kraft.
\newblock On degree bounds for separating invariants.
\newblock {\em Preprint, arXiv:1001.5216}, 2010.

\bibitem{Mr10}
Mara~D. Neusel and M{\"u}fit Sezer.
\newblock Separating invariants for modular {$p$}-groups and groups acting
  diagonally.
\newblock {\em Math. Res. Lett.}, 16(6):1029--1036, 2009.

\bibitem{MR2556140}
M{\"u}fit Sezer.
\newblock Constructing modular separating invariants.
\newblock {\em J. Algebra}, 322(11):4099--4104, 2009.

\bibitem{Comb}
M{\"u}fit Sezer.
\newblock Explicit separating invariants for cyclic $p$-groups.
\newblock {\em J. Combin. Theory Ser. A}, (doi:10.1016/j/jcta.2010.05.003),
  2010.

\bibitem{MR1639884}
R.~James Shank.
\newblock S.{A}.{G}.{B}.{I}. bases for rings of formal modular seminvariants
  [semi-invariants].
\newblock {\em Comment. Math. Helv.}, 73(4):548--565, 1998.

\bibitem{weh}
D.~L. Wehlau.
\newblock Invariants for the modular cyclic group of prime order via classical
  invariant theory.
\newblock {\em Preprint, arXiv:0912.1107v1}, 2009.

\end{thebibliography}
 \end{document}